\documentclass[smallextended,referee,envcountsect]{svjour3} 
\smartqed 
\UseRawInputEncoding
\usepackage{graphicx}
 \usepackage{amscd}
 \usepackage{amssymb}
 \usepackage{hyperref}
 \usepackage[all]{xypic}
 \usepackage{mathrsfs}
 \usepackage{mathtools}
 \usepackage{amsmath}
\journalname{JOTA}

\begin{document}

\title{Global optimization on a metric space with a graph and an application to PBVP}


\author{Abhik Digar   \and  G. Sankara Raju Kosuru }

\institute{Department of Mathematics\\
             Indian Institute of Technology Ropar \\
             Rupnagar - 140 001, Punjab, India.\\
             abhikdigar@gmail.com \and raju@iitrpr.ac.in  
}

\date{Received: date / Accepted: date}

\maketitle

\begin{abstract}

In this article we introduce a new type of cyclic contraction mapping on a pair of subsets of a metric space with a graph and prove best proximity points results for the same. Also, we demonstrate that the number of such points is same with the number of connected subgraphs. Hereafter, we introduce a fixed point mapping obtained from the aforesaid cyclic contraction and prove some fixed point theorems which will be used to find a common solution for a system of periodic boundary value problems. Our results unify and subsume many existing results in the literature

%
\end{abstract}
\keywords{best proximity point \and fixed point \and G-cyclic contraction \and metric space with graph \and periodic boundary value problem}
\subclass{47H10 \and  34B15 \and 54H25}


\section{Introduction and Preliminaries}


Let $(X,d)$ be a metric space and $G$ be a directed graph endowed on $X$ with the vertex set $V(G)=X$ and the edge set $E(G).$ Assume that every trivial loop is an edge of $G$ (i.e., $(x,x)\in E(G)~\mbox{for}~x\in X)$ and $G$ doesn't contain any parallel edge.
 In this case, we call such a triple $(X,d,G)$ as a metric space with a graph.  
 The edge set $E(G)$ is said to have quasi-order if $(x,y), (y,z)\in E(G)$ for $x, y, z\in X,$ then $(x,z)\in E(G)$ (\cite{Jachymski 2008}). For any vertices $x$ and $y$ in $G,$ a path from $x$ to $y$ is a finite sequence of vertices $x=x_0, x_1,...,x_n=y$ in $G$ with $(x_{i-1}, x_i)\in E(G)$ for $1\leq i\leq n$, for some $n\in \mathbb{N}.$
   Also the graph $G$ is said to be connected if for every pair of vertices $x$ and $y,$ there exists a path from $x$ to $y$ in $G$. The graph $G$ is known as weakly connected if the undirected graph $\tilde{G}$ is connected, where $V(\tilde{G})=V(G)$ and $E(\tilde{G})=E(G)\bigcup \{(y,x):(x,y)\in E(G)\}$). 
For any vertex $x$ in $G,~[x]_{\tilde{G}}=\{y\in V(G):~\mbox{there is a path from}~x~\mbox{to}~y$ in $\tilde{G}\}.$  
Further $A$ is said to have the property $(*)$ if for any sequence $\{x_n\}$ in $A$ with $\displaystyle \lim_{n\to \infty} x_n=x$ and $(x_n,x_{n+1})\in E(G),$ we have $(x_n,x)\in E(G)$ for all $n\in \mathbb{N}$ ([1]). It is proved in [1] that $E(G)$ is a quasi-order if the triple $(X,d,G)$ has property $(*)$.
By using a few geometrical notions and properties of $G$, in [1-7]
 the authors obtained the existence of a fixed point for certain type of mappings on $X$.  
Let $A, B$ be two subsets of $X$ and $T:A\cup B\to A\cup B$ be a cyclic map, i.e., $T(A)\subseteq B, T(B)\subseteq A.$ If $d(A,B) = \inf\{d(a,b):a\in A,b\in B\} >0$, then 
a best proximity point ($x\in A\cup B$ that satisfies $d(x,Tx)=d(A,B)$) serves as a global optima of the operator equation $Tx=x$.

   To address such a global optimization problem, in this article, we develop an approach that integrates the compatibilty between the structural properties of $G$ and the geometrical properties on $X.$ 
 More precisely, we introduce $G$-cyclic $(\phi_1, \phi_2)$-contraction mappings and therein prove the existence of a best proximity point, here $\phi_1,~\phi_2$ are such that $\phi_1$ is an increasing map and $\phi_2-I$ is a non-decreasing map for the identity map $I$. Our results ensures the uniqueness of such a point in every connected subgraph of $G.$ We also yield some results on the relation between the cardinality of best proximity points and the number of connected subgraphs in $G$. Further, we obtain common fixed point results for a similar class of mappings. Our results unify and subsume a few existing results in the literature. Finally, as an application, we obtain the existence of a common solution for a system of periodic boundary value problems. As a particular case, we get the existence of a unique solution for a periodic boundary value problem with Carath\'{e}odory function on right hand side.

 Now we fix a few notations and facts that will be used in the sequel. Throughout this paper  $I$ denotes the identity map on $[0, \infty).$ Let $(X,d)$ be a metric space and $G$ be a directed graph endowed on $X$ with the vertex set $V(G)=X$ and the edge set $E(G).$ Assume that every trivial loop is an edge of $G$ (i.e., $(x,x)\in E(G)~\mbox{for}~x\in X)$ and $G$ doesn't contain any parallel edge.
 In this case, we call such a triple $(X,d,G)$ as a metric space with a graph.  
 The edge set $E(G)$ is said to have quasi-order if $(x,y), (y,z)\in E(G)$ for $x, y, z\in X,$ then $(x,z)\in E(G)$ (\cite{Jachymski 2008}). For any vertices $x$ and $y$ in $G,$ a path from $x$ to $y$ is a finite sequence of vertices $x=x_0, x_1,...,x_n=y$ in $G$ with $(x_{i-1}, x_i)\in E(G)$ for $1\leq i\leq n$, for some $n\in \mathbb{N}.$
   Also the graph $G$ is said to be connected if for every pair of vertices $x$ and $y,$ there exists a path from $x$ to $y$ in $G$. The graph $G$ is known as weakly connected if the undirected graph $\tilde{G}$ is connected, where $V(\tilde{G})=V(G)$ and $E(\tilde{G})=E(G)\bigcup \{(y,x):(x,y)\in E(G)\}$). 
For any vertex $x$ in $G,~[x]_{\tilde{G}}=\{y\in V(G):~\mbox{there is a path from}~x~\mbox{to}~y$ in $\tilde{G}\}.$  
Further $A$ is said to have the property $(*)$ if for any sequence $\{x_n\}$ in $A$ with $\displaystyle \lim_{n\to \infty} x_n=x$ and $(x_n,x_{n+1})\in E(G),$ we have $(x_n,x)\in E(G)$ for all $n\in \mathbb{N}$ ([1]). It is proved in [1] that $E(G)$ is a quasi-order if the triple $(X,d,G)$ has property $(*)$.
For any two subsets $A, B$ of $X$, we set $A_0 =\{x\in A : d(x,y)=d(A,B) \mbox{ for some } y \in B\},~B_0 =\{v\in B : d(u,v)=d(A,B) \mbox{ for some } u \in A\}$ and for $(x, y)\in A\times B,~m(x,y)=\max \{d(x, Tx), d(y, Ty)\}.$ The pair $(A,B)$ is said to be sharp proximal  
 (\cite{Espinola 2008,Raju 2010}),
  if for any $x$ in $A$ (respetively in $B$) there exists a unique $y$ in $B$ (respetively in $A$) such that $d(x,y)=d(A,B)$. Also, we say that the pair $(A,B)$ is $G$-Chebyshev if every parallel pair is also an edge of $G$ (i.e., $\{(u,v)\in A\times B:d(u,v)=d(A,B)\} \subset E(G)$).
The pair $(A,B)$ is said to satisfy property UC (\cite{Suzuki 2009}) if for any sequence $\{x_n\}, \{u_n\}$ in $A$ and $\{y_n\}$ in $B$ such that $\displaystyle \lim_{n\to \infty} d(x_n, y_n)=d(A,B)=\lim_{n\to \infty} d(u_n, y_n),$ then $\displaystyle \lim_{n\to \infty} d(x_n,u_n)= 0.$  The following well known lemma is useful to obtain some of our main results.
   \begin{lemma}\label{Suzuki lemma}\cite{Suzuki 2009}
 Let $A$ and $B$ be two subsets of a metric space $(X,d).$ Assume that $(A,B)$ has property UC. Let $\{x_n\}$ and $\{y_n\}$ be two sequences in $A$ and $B$ respectively, such that either of the following two holds:
 $$\displaystyle \lim_{m\to \infty}\sup_{n\geq m}d(x_m,y_n)=d(A,B)~~~~\mbox{or}~~~~\displaystyle \lim_{n\to \infty}\sup_{m\geq n}d(x_m,y_n)=d(A,B).$$
 Then $\{x_n\}$ is Cauchy.
 \end{lemma}

\section{$G$-cyclic contractions}
%

Let $A,B$ be two subsets of a metric space with a graph $(X,d,G)$ and $T$ be a cyclic map on $A\cup B.$ We say that $T$ (respectively, $T^2$) preserves edges on $A$ if $(Tu,Tv)$ (respectively, $(T^2u,T^2v)$) in an edge for every $u, v\in A$ with $(u, v)\in E(G).$

  \begin{definition}\label{map definition_pre}
Let $(A,B)$ be a $G$-Chebyshev pair of subsets of a metric space with a graph $(X,d,G)$. Suppose $\phi_1:[0,\infty)\to [0,\infty)$ is an increasing map and $\phi_2:[0,\infty)\to [0,\infty)$ is a map such that $\phi_2-I$ is non-decreasing, here $I$ is the identity map. A cyclic map $T$ on $A\cup B$ is said to be $G$-cyclic $(\phi_1, \phi_2)$-contraction
  if 
 \begin{itemize}
 \item[(i)] $T^2$ preserves the edges on $A;$
  \item[(ii)] $d(Tx,Ty)\leq (I-\phi_1)(d(x,y))+ (I-\phi_2)(m(x,y))+ (\phi_1+\phi_2-I)(d(A,B))$ for $\displaystyle (x,y)\in A\times B$ with $\{(x,y),(x,Ty),(Ty,x)\} \cap E(G) \neq \emptyset.$
 \end{itemize}
 \end{definition}
%
 It is easy to see that if $T$ is a $G$-cyclic $(\phi_1, \phi_2)$-contraction, then $T(A_0)\subseteq B_0, T(B_0)\subseteq A_0$.
 Also, every cyclic $\phi$-contraction (\cite{Shahzad 2009}) is a $G_0$-cyclic $(\phi, c+I)$-contraction, where $E(G_0)=X\times X$ and $c$ is a constant. Further, a cyclic contraction (\cite{Eldred 2006}) with a constant $\alpha \in (0,1)$ is a $G_0$-cyclic $(\phi, c+I)$-contraction, where $\phi (s)=(1-\alpha)s$ and $c$ is a constant.  However, a $G$-cyclic $(\phi, c+I)$-contraction is not necessarily a cyclic $\phi$-contraction.
  
 \begin{example}\label{not G-contraction 3}
 Consider the space $X=\{f=f_1+if_2: f_1,f_2\in \mathscr{C}\left([0,1]\right)\}$ with $\|f\|=\|f_1\|_{\infty}+\|f_2\|_{\infty}$ for $f=f_1+f_2\in X.$  Set
 $A=\{f_\alpha\in X: \alpha\in [0,1]\},~B=\{g_\alpha\in X: \alpha\in [0,1]\},$ where 
{\small{ 
 \[
 	f_{\alpha}(t)=
 		\begin{cases}
 		2i\alpha t   &\text{if}~ t\in \left[0,\frac{1}{2}\right],\\
 		2i\alpha(1-t) &\text{if}~ t\in \left[\frac{1}{2},1\right];
 		\end{cases}~~
 \\	
 	g_{\alpha}(t)=
 		\begin{cases}
 		1+\alpha(t-\frac{1}{2})+ 2i\alpha t   &\text{if}~ t\in \left[0,\frac{1}{2}\right],\\
 		1-\alpha(t-\frac{1}{2})+ 2i\alpha (1-t)   &\text{if}~ t\in \left[\frac{1}{2},1\right].
 		\end{cases}
 \]}}
 Fix $\alpha\in [0,1].$ Then $\|f_\alpha\|=\alpha,~\|g_\alpha\|=1+\alpha.$ For $t\in [0,\frac{1}{2}],~|f_\alpha(t)-g_\alpha(t)|=|2i\alpha t-(1+\alpha(t-\frac{1}{2})+ 2i\alpha t)|=|1+\alpha(t-\frac{1}{2})|$ and for $t\in [\frac{1}{2},1],~|f_\alpha(t)-g_\alpha(t)|=|2i\alpha(1-t)-(1-\alpha(t-\frac{1}{2})+ 2i\alpha (1-t))|=|1-\alpha(t-\frac{1}{2})|.$ Then $\|f_\alpha-g_\alpha\|=1.$
Now, for $\alpha, \beta\in [0,1]$ with $\alpha\neq \beta$ and $t\in [0,\frac{1}{2}],~|f_\alpha(t)-g_\beta(t)|=|2i\alpha t-(1+\beta(t-\frac{1}{2})+ 2i\beta t)|=|1+\beta(t-\frac{1}{2})|+2|\alpha-\beta|t$ and for $t\in [\frac{1}{2},1],~|f_\alpha(t)-g_\beta(t)|=|2i\alpha(1-t)-(1-\beta(t-\frac{1}{2})+ 2i\beta (1-t))|=|1-\beta(t-\frac{1}{2})|+2|\alpha-\beta|(1-t).$ Then $\|f_\alpha-g_\beta\|=1+|\alpha-\beta|>1$ for $\alpha\neq \beta.$

It's clear that for every $(f_\alpha, g_\beta)\in A\times B$ there exists a there exists a unique pair $(g_\alpha, f_\beta)\in B\times A$ such that $\|f_\alpha-g_\alpha\|=1=d(A,B)=\|g_\beta-f_\beta\|.$ Hence $(A,B)$ is a sharp proximal pair.
For any $0<z< 1,$ there exists a unique $n\in \mathbb{N}\cup \{0\}$ such that $\frac{1}{n+1}\leq z<\frac{1}{n}.$ Denote $n+1$ by $\kappa_z.$ Define $k_0=0$ and $k_1=1.$ Consider a graph $G$ on $X$ by $V(G)=X$ and $E(G)=\{(f_\alpha, g_\beta)\in A\cup B\times A\cup B: \alpha\geq \beta~\mbox{and}~ \kappa_{\alpha}=\kappa_{\beta}~\mbox{for}~\alpha, \beta\in (0,1)\}\cup \{(f_0, g_0), (f_1, g_1)\}.$
For $\alpha\in [0,1],$~define a map $T:A\cup B\to A\cup B$ by
 \[
 	Tf_\alpha=
 		\begin{cases}
 		g_{\frac{1}{\kappa_{\alpha}}}~&\text{if}~0< \|f_\alpha\|<1,\\
 		g_\alpha &\text{otherwise};
 		\end{cases}~
 	Tg_\alpha=	
 		\begin{cases}
 		f_{\frac{1}{\kappa_{\alpha}}}~&\text{if}~0< \|g_\alpha\|-1<1,\\
 		f_\alpha &\text{otherwise}.
 		\end{cases}
 \]
 Define $\phi:[0,\infty)\to [0,\infty)$ by $\phi(s)=\lfloor s\rfloor +\frac{\{s\}}{\kappa_{\{s\}}},$ where $\lfloor s\rfloor$ is the floor function and $\{s\}$ is the fractional part of $s.$ It is easy to verify that $\phi$ is an increasing function.
We first prove that $T$ is a $G$-cyclic $(\phi, c+I)$-contraction map, where $c$ is a constant. Certainly, $\|Tf_0-Tg_0\|=\|g_0-f_0\|=1$ and $\|Tf_1-Tg_1\|=\|g_1-f_1\|=1.$
Next, for $\alpha , \beta \in (0,1)$ and $t\in [0,\frac{1}{2}],$ 
 \begin{eqnarray*}
 \left|Tf_\alpha(t)-Tg_\beta(t)\right|
  &=&\left|g_{\frac{1}{\kappa_{\alpha}}}(t)-f_{\frac{1}{\kappa_{\beta}}}(t)\right|\\
&=& \left|\left(1+\frac{1}{\kappa_\alpha}\left(t-\frac{1}{2}\right)+2i\frac{1}{\kappa_\alpha}t\right)-2i\frac{1}{\kappa_\beta}t \right|\\
&=& \left|1+\frac{1}{\kappa_\alpha}\left(t-\frac{1}{2}\right) +2i\left(\frac{1}{\kappa_\alpha}-\frac{1}{\kappa_{\beta}}\right)t\right|.
 \end{eqnarray*}
 Similarly, $\left|Tf_\alpha(t)-Tg_\beta(t)\right|$=$\left|1-\frac{1}{\kappa_\alpha}\left(t-\frac{1}{2}\right) +2i\left(\frac{1}{\kappa_\alpha}-\frac{1}{\kappa_{\beta}}\right)(1-t)
\right|$ for $t \in [\frac{1}{2},1].$ 
 Thus, $\|Tf_\alpha-Tg_\beta\|=1+\left|\frac{1}{\kappa_\alpha}-\frac{1}{\kappa_\beta}\right|.$ Also,
 {\small{\begin{eqnarray*}
 \|f_\alpha-g_\beta\|-\phi\left(\|f_\alpha-g_\beta\|\right)+\phi(d(A,B))&= & 1+|\alpha-\beta|-\phi\left(1+|\alpha-\beta|\right)+\phi(1)\\
 &=& 1+|\alpha-\beta|-\left(1+\frac{|\alpha-\beta|}{\kappa_{|\alpha-\beta|}}\right)+1\\
 &=& 1+|\alpha-\beta|\left(1-\frac{1}{\kappa_{|\alpha-\beta|}}\right)
 \end{eqnarray*}}}
  As $(f_\alpha, g_\alpha)\in E(G),$ we have $\|Tf_\alpha-Tg_\alpha\|=1=\|f_\alpha-g_\alpha\|=\|f_\alpha-g_\alpha\|-\phi\left(\|f_\alpha-g_\alpha\|\right)+\phi(d(A,B)).$
 
  If $(f_\alpha, g_\beta)\in E(G)$ for $\alpha, \beta\in (0,1),~\alpha\neq \beta,$ we have $\kappa_\alpha=\kappa_\beta.$ Then $\|Tf_\alpha-Tg_\beta\|=1+\left|\frac{1}{\kappa_\alpha}-\frac{1}{\kappa_\beta}\right|=1\leq 1+|\alpha-\beta|-\left(1+\frac{|\alpha-\beta|}{\kappa_{|\alpha-\beta|}}\right)+1=\|f_\alpha-g_\beta\|-\phi\left(\|f_\alpha-g_\beta\|\right)+\phi(d(A,B)).$ 
Again, if $(Tg_\beta, f_\alpha)\in E(G)$ or $(f_\alpha, Tg_\beta)\in E(G)$ for $\alpha, \beta\in (0,1),~\alpha\neq \beta,$ we have $\alpha=\frac{1}{\kappa_\beta}.$ Then $\kappa_\alpha =\kappa_\beta$ and hence $\|Tf_\alpha-Tg_\beta\|\leq \|f_\alpha-g_\beta\|-\phi\left(\|f_\alpha-g_\beta\|\right)+\phi(d(A,B)).$

  Let $\alpha_0=0.49$ and $\beta_0=0.51$ so that $\kappa_{\|f_{\alpha_0}\|}=\kappa_{\alpha_0}=3$ and $\kappa_{\|g_{\beta_0}\|-1}=\kappa_{\beta_0}=2.$ We see that 
 $\|Tf_{\alpha_0}-Tg_{\beta_0}\|=1+\left|\frac{1}{\kappa_{\alpha_0}}-\frac{1}{\kappa_{\beta_0}} \right|=1+\left|\frac{1}{3}-\frac{1}{2}\right|=1+\frac{1}{6}> 1.02=1+|\alpha_0-\beta_0|=\|f_{\alpha_0}-g_{\beta_0}\|.$ Thus, $T$ is not a cyclic $\phi$-contraction for any increasing map $\phi.$
 \end{example}

 The set $X_{T^2}^A$ denotes the set of all vertices $x$ in $A$ for which $(x,T^2x)$ is an edge in $G.$
 The following approximation lemma will be useful in the sequel.
   \begin{lemma}\label{abps1}
  Let $(X,d,G)$ be a metric space with a graph and let $(A,B)$ be a pair of non-empty subsets of $X.$ Suppose $\phi_i:[0,\infty)\to [0,\infty), ~(1\leq i\leq 2),$ is such that $\phi_1$ increasing and $\phi_2-I$ is non-decreasing. Assume that $T:A\cup B\to A\cup B$ is a $G$-cyclic $(\phi_1, \phi_2)$-contraction. Then for any $x_0\in X_{T^2}^A,~d(x_{n+1}, x_n)\to d(A,B),$ where $x_{n+1}=Tx_n, n\geq 0$.
 \end{lemma}
 
 \begin{proof}
 As $x_0 \in X^A_{T^2}$, we have $(x_{n},x_{n+2}) \in E(G)$ for even $n$ and $(x_n,x_n) \in E(G)$ for odd $n.$ For $n\geq 0,$ set $d_{n}= d(x_{n+1}, x_n).$
 It has to be observe that $(\phi_2-I)(m(x_{n+1}, x_n))\geq  (\phi_2-I)(d_{n})$ for $n\geq 0.$
%
Since $\phi_1$ is increasing and $\phi_2-I$ is non-decreasing, for $n\geq 0$ we have
\begin{eqnarray}\label{for lemma abps1}
d_{n+1}&= & d(x_{n+2},x_{n+1}) \nonumber\\ 
& \leq & (I-\phi_1) (d(x_{n+1},x_{n}))+(I-\phi_2)(m(d(x_{n+1},x_{n}))+(\phi_1+\phi_2-I)(d(A,B)) \nonumber\\
&\leq & (I-\phi_1) (d(x_{n+1},x_{n}))+(I-\phi_2)(d(x_{n+1},x_{n}))+(\phi_1+\phi_2-I)(d(A,B)) \nonumber \\
&= & d_n -\phi_1(d_n) -(\phi_2-I) (d_n)+\phi_1(d(A,B))+(\phi_2-I) (d(A,B))\\
&\leq & d_n. \nonumber
\end{eqnarray}
%
%
%
Hence $\{d_n\}$ is non-increasing with $d(A,B)\leq d_n\leq d_0.$ Suppose $\displaystyle \lim_{n\to \infty}d_n=p.$ Then $d(A,B)\leq p\leq d_n$ for $n\geq 0.$ From (\ref{for lemma abps1}), it is clear that $(\phi_1+\phi_2-I)(d_{n})\leq d_{n}-d_{n+1}+(\phi_1+\phi_2-I)(d(A,B))$ for $n\geq 0.$ Thus $(\phi_1+\phi_2-I)(d(A,B)) \leq  (\phi_1+\phi_2-I)(p) \leq (\phi_1+\phi_2-I)(d_{n}) \leq  d_{n}-d_{n+1}+(\phi_1+\phi_2-I)(d(A,B)).$ Letting $n\to \infty,~(\phi_1+\phi_2-I)(p) \leq (\phi_1+\phi_2-I)(d(A,B)).$
As $\phi_1+\phi_2-I$ is increasing, we must have $p=d(A,B).$ This completes the proof.
 \qed \end{proof}
 
 
\section{Existence of Best Proximity Points}

 Let $(X,d,G)$ be a metric space with a graph and let $A, B$ be two non-empty subsets of $X$. Suppose $\phi_i:[0,\infty)\to [0,\infty),~(1\leq i\leq 2),$ are maps such that $\phi_1$ is increasing and $\phi_2-I$ is non-decreasing. Suppose $T$ be a $G$-cyclic $(\phi_1, \phi_2)$-contraction on $A\cup B.$ We denote the set of all best proximity points of $T$ in $A$ by $BP\left(T\big|_A\right)$. 
 The proof of the following proposition is straightforward and hence left to the reader as an exercise. 
    \begin{proposition}\label{proposition1}
Let $A, B, T$ and $(X,d,G)$ be as above. Suppose that $(A,B)$ is having property UC. Then $BP\left(T\big|_A\right)\subseteq X_{T^2}^A.$ Moreover, $T^{2n}x=x$ for all $x\in BP\left(T\big|_A\right)$ and $n\in \mathbb{N}.$ 
 \end{proposition}
   
 Proposition \ref{proposition Cauchy} will be instrumental in proving some of our main results. 

  \begin{proposition}\label{proposition Cauchy}
Let $A, B, T$ and $(X,d,G)$ be as in Lemma \ref{abps1}. Suppose $(A,B)$ has property UC and $A$ has property $(*)$. Then $\{x_{2n}\}$ is a Cauchy sequence in $A$ for any $x_0\in X_{T^2}^A$ and $x_{n+1}=Tx_n, n\geq 0$.
 \end{proposition}
 
 \begin{proof}
 Fix $x_0\in X^{A}_{T^2}.$ By Lemma \ref{abps1}, we have $$d(x_{2n}, x_{2n+1})\to d(A,B)~\mbox{and}~d(x_{2n+1}, x_{2n+2})\to d(A,B).$$ Hence by property UC, $d(x_{2n}, x_{2n+2})\to 0.$\\

{\it{Assertion:}} For each $\epsilon>0,$ there exists $N\in \mathbb{N}$ such that 
 \begin{eqnarray*}\label{to be proved}
 \mbox{for}~ p>n\geq N,~ d(x_{2p}, x_{2n+1})<d(A,B)+\epsilon.
 \end{eqnarray*}

  {\it{Proof of the assertion:}} Assume the contratry. Then there exists $\epsilon_0>0$ such that for each $j\geq 1,$ there exist $p_j> n_j\geq j$ such that 
  \begin{eqnarray*}
  d(x_{2p_j}, x_{2n_j+1})&\geq & d(A,B)+\epsilon_0\\ \mbox{and}~~
  d(x_{2(p_j-1)}, x_{2n_j+1})&< & d(A,B)+\epsilon_0.
  \end{eqnarray*}
 Then 
 \begin{eqnarray*}
 d(A,B)+\epsilon_0 &\leq & d(x_{2p_j}, x_{2n_j+1})\\
 &\leq & d(x_{2p_j}, x_{2p_j-2})+d(x_{2p_j-2}, x_{2n_j+1})\\
 &<& d(x_{2p_j}, x_{2p_j-2}) + (d(A,B)+\epsilon_0).
 \end{eqnarray*}  
 As $j\to \infty,$ we have $\displaystyle\lim_{j\to \infty} d(x_{2p_j}, x_{2n_j+1})\to d(A,B)+\epsilon_0.$ Since $A$ has property $(*),$
 $(x_{2n_j+2}, x_{2p_j})\in E(G), (x_{2n_j+2}, x_{2p_j+2})\in E(G),$ we have 
 \begin{eqnarray*}
 d(x_{2p_j}, x_{2n_j+1}) &\leq & d(x_{2p_j}, x_{2p_j+2})+d(x_{2p_j+2}, x_{2n_j+3})+d(x_{2n_j+3}, x_{2n_j+1})\\
 &\leq & d(x_{2p_j}, x_{2p_j+2})+ d(x_{2p_j+1}, x_{2n_j+2})+d(x_{2n_j+3}, x_{2n_j+1})\\
 &\leq & d(x_{2p_j}, x_{2p_j+2})+[(I-\phi_1)(d(x_{2p_j},  x_{2n_j+1}))+(I-\phi_2)(m(x_{2p_j}, x_{2n_j+1})\\ 
 && +(\phi_1+\phi_2-I)(d(A,B))]+d(x_{2n_j+3}, x_{2n_j+1}).\\
 &\leq & d(x_{2p_j}, x_{2p_j+2})+ d(x_{2p_j},  x_{2n_j+1})+ d(x_{2n_j+3}, x_{2n_j+1}).
 \end{eqnarray*} 
 Letting $j\to \infty,$ we obtain
 \begin{eqnarray*}
 d(A,B)+\epsilon_0 &\leq & d(A,B)+\epsilon_0 -\displaystyle \lim_{j\to \infty}[\phi_1(d(x_{2p_j}, x_{2n_j+1}))- (I-\phi_2)(m(x_{2p_j}, x_{2n_j+1}))]\\
 && +(\phi_1+\phi_2-I)(d(A,B)) ~~\leq d(A,B)+\epsilon_0.
 \end{eqnarray*}  
 Thus, $\displaystyle \lim_{j\to \infty}(\phi_1(d(x_{2p_j}, x_{2n_j+1}))+ (\phi_2-I)(m(x_{2p_j}, x_{2n_j+1})))= (\phi_1+\phi_2-I)(d(A,B)).$
  Using the fact that $\phi_1$ is increasing and $\phi_2-I$ is non-decreasing, we see that
 \begin{eqnarray*}
 \phi_1(d(A,B)+\epsilon_0)+(\phi_2-I)(d(A,B)) &\leq & \displaystyle \lim_{j\to \infty} [\phi_1(d(x_{2p_j}, x_{2n_j+1}))+(\phi_2-I)(m(x_{2p_j}, x_{2n_j+1}))]\\ &= & (\phi_1+\phi_2-I)(d(A,B)),
 \end{eqnarray*}
%
 which is absurd. This proves the assertion. Now by Lemma \ref{Suzuki lemma}, we can conclude that $\{x_{2n}\}$ is a Cauchy sequence.
 \qed \end{proof}
  
 Two sequences $\{x_n\}$ and $\{y_n\}$ in $A$ are known as Cauchy equivalent (\cite{Jachymski 2008}) if each of them is Cauchy and $d(x_n,y_n)\to 0$.
  The following theorem characterizes 
 $BP\left(T\big|_A\right)$ in light of the connectedness of $G$.
 \begin{theorem}\label{equivalence theorem}
 Let $(X,d,G)$ and $(A,B)$ be as in Proposition \ref{proposition Cauchy}. Suppose $(A,B)$ is a sharp proximal pair. Then the following are equivalent:
 \begin{itemize}
 \item[(a)] $G$ in A is weakly connected;
 \item[(b)] For any $G$-cyclic $(\phi_1, \phi_2)$-contraction $T$ on $A\cup B$, the sequences $\{T^{2n}x\}$ and $\{T^{2n}y\}$ are Cauchy equivalent for $x, y\in A$;
 \item[(c)] For any $G$-cyclic $(\phi_1, \phi_2)$-contraction $T$ on $A\cup B$, the cardinality of $BP\left(T\big|_A\right)\leq 1.$
 \end{itemize}
 \end{theorem}
 
 \begin{proof}
 $(a)\Rightarrow (b):$ Let $T:A\cup B\to A\cup B$ be a $G$-cyclic $(\phi_1, \phi_2)$-contraction and $x, y\in A$. Assume $G$ in $A$ is weakly connected. For any $n\geq 0,$ define $d'_n=d(T^{n+1}x,T^ny).$ Since $A$ has property $(*)$, we have $(x,y),$ $(T^{2n}x, T^{2n}y),$ $(T^{2n}y, T^{2n+2}x)\in E(G)$ for $n\geq 0.$ Then  
 \begin{eqnarray}\label{equivalence number}
 d'_{n+1}& = & d(T^{n+2}x,T^{n+1}y)\nonumber\\ 
 &=& d(TT^{n+1}x,TT^{n}y)\nonumber\\
  &\leq & (I-\phi_1)(d(T^{n+1}x,T^{n}y))+ (I-\phi_2)(m(T^{n+1}x,T^{n}y))\nonumber\\ &+& (\phi_1+\phi_2-I)(d(A,B))\\ 
 &\leq &(I-\phi_1)(d'_{n})+(I-\phi_2)(d(A,B))+(\phi_1+\phi_2-I)(d(A,B))\leq d'_n.\nonumber
 \end{eqnarray}
 Therefore, $\{d'_n\}$ is non-increasing and $d(A,B)\leq d'_n\leq d(y,Tx)$ for $n\geq 0.$ Suppose $\displaystyle \lim_{n\to \infty}d'_n=p.$ Then $d(A,B)\leq p\leq d'_n$ for $n\geq 0.$ Using (\ref{equivalence number}), we have
 \begin{eqnarray*}
 (\phi_1+\phi_2-I)(d(A,B))&\leq & (\phi_1)(p)+(\phi_2-I)(d(A,B))\\ &\leq & (\phi_1)(d'_n)+(\phi_2-I)(m(T^{n+1}x,T^{n}y))\\ & \leq & d'_n-d'_{n+1}+(\phi_1+\phi_2-I)(d(A,B)).
 \end{eqnarray*}
 Thus $\phi_1(p)= \phi_1(d(A,B))$.
Since $\phi_1$ is increasing, we obtain $p=d(A,B).$ From Proposition \ref{proposition Cauchy} it is clear that $\{T^{2n}x\}$ and $\{T^{2n}y\}$ are Cauchy sequences. By using the above technique and Lemma \ref{abps1}, we have $d'_{2n+1}=d(T^{2n+2}x,T^{2n+1}y)\to d(A,B)$ and $d(T^{2n+2}y,T^{2n+1}y) \to d(A,B)$ and $n\to \infty$. Using property UC, we get
   $\{T^{2n}x\}$ and $\{T^{2n}y\}$ are Cauchy equivalent.
   
 $(b)\Rightarrow (c):$ Let $x, y\in BP\left(T\big|_A\right).$ Then $x, y\in X_{T^2}^A$. By $(b),~d(T^{2n}x,T^{2n}y)\to 0$ as $n\to \infty.$ Proposition \ref{proposition1} ensures that $T^{2n}x=x$ and $T^{2n}y=y$ for $n\geq 0.$ Hence $x=y.$
 
 $(c)\Rightarrow (a):$ On the contrary, assume that $G$ in $A$ is not weakly connected. Let $p_1\in A,$ then certainly $[p_1]_{\tilde{G}}\cap A$ and $A\setminus [p_1]_{\tilde{G}}$ are non-empty disjoint sets. Since $(A,B)$ is sharp proximal, there exists $q_1\in B$ such that $d(p_1,q_1)=d(A,B).$ Also for $p_2\in A\setminus [p_1]_{\tilde{G}},$ there exists $q_2\in B$ such that $d(p_2,q_2)=d(A,B).$ Observe that $p_1\in [q_1]_{\tilde{G}}, q_2\in [p_2]_{\tilde{G}}$ and $p_2\notin [p_1]_{\tilde{G}}$. Hence $q_2\notin [p_1]_{\tilde{G}}=[q_1]_{\tilde{G}}.$  Define a map $T:A\cup B\to A\cup B$ by 
 \[
 T(x)=
 	\begin{cases}
 	q_1 & \text{for}~ x\in [p_1]_{\tilde{G}}\cap A;\\
 	q_2 & \text{for}~ x\in A\setminus [p_1]_{\tilde{G}}
 	\end{cases};~~
 T(y)=
 	\begin{cases}
 	p_1 & \text{for}~ y\in [q_1]_{\tilde{G}}\cap B;\\
 	p_2 & \text{for}~ y\in B\setminus [q_1]_{\tilde{G}}.
 	\end{cases}
 \] 
 Notice that for a constant $k,~T$ is a $G$-cyclic $(\phi,k+I)$-contraction on $A\cup B$ with $\phi (s)=\frac{1}{4}s$ for all $s\geq 0.$ Here $BP\left(T\big|_A\right)=\{p_1, p_2\}.$ This disobeys $(c)$.
 \qed \end{proof}
 
 \begin{remark}
In the above theorem, it is to be observed that, sharp proximality on $(A,B)$ has been used only to prove $(c)\Rightarrow (a)$ and can be relaxed in proving $(a)\Rightarrow (b)$ and $(b)\Rightarrow (c)$.
 \end{remark}

  The following theorem ensures the existence of a best proximity point for $T.$

 \begin{theorem}\label{theorem for subsequence}
Let $A, B, T$ and $(X,d,G)$ be as in Lemma \ref{abps1}. Suppose $x_0\in X_{T^2}^A$ and $x_{n+1}=Tx_n, n\geq 0.$ If $A$ has property $(*)$ and $\{x_{2n}\}$ has a convergent subsequence in $A,$ then $T$ has a best proximity point.  
 \end{theorem}
 
 \begin{proof}
 Choose a subsequence $\{x_{2n_k}\}$ of $\{x_{2n}\}$ such that $x_{2n_k}\to y$ in $A$. By property $(*),$ we have $(x_{2{n_k}},y)\in E(G)$ for $k\geq 1.$ Then 
 \begin{eqnarray*}
 d(A,B)\leq d(y,Ty)&\leq & d(y, x_{2{n_k}+2})+ d(Ty, x_{2{n_k}+2})\\
 & \leq & d(y, x_{2{n_k}+2})+ d(y, x_{2{n_k}+1})\\
 &\leq & d(y, x_{2{n_k}+2})+d(y, x_{2{n_k}})+d(x_{2{n_k}}, x_{2{n_k}+1})\\
 &\to & d(A,B)~\mbox{as}~k\to \infty~(\mbox{by Lemma \ref{abps1}}).
 \end{eqnarray*}
 Therefore, $d(Ty, y)=d(A,B).$
 \qed \end{proof}
 
One can see that Theorem 4 of \cite{Shahzad 2009} and Proposition 3.2 of \cite{Eldred 2006} are particular case of Theorem \ref{theorem for subsequence}. For this one can consider $T$ is a $G_0$-cyclic $(\phi, c+I)$- contraction, where $V(G_0)=X, E(G_0)=X\times X,~c$ is a constant; and particularly in the later case, for some $\alpha \in (0,1),~\phi (s)=(1-\alpha)s, \forall s\in [0,\infty).$
For a subset $S$ of $A,$ the map $T|_S:S\to B$ is said to be a best proximity operator (abbrev. BPO) on $S$ if there is a unique point $x^*$ in $S$ such that $d(x^*, Tx^*)=d(A,B)$ and $\displaystyle \lim_{n\to \infty}T^{2n}x=x^*$ for any $x\in S.$ The following result is a main best proximity theorem for a $G$-cyclic $(\phi_1, \phi_2)$-contraction in presence of property UC.

 \begin{theorem}\label{main theorem1}
  Let $A, B, T$ and $(X,d,G)$ be as in Lemma \ref{abps1}. Suppose $A$ is complete, has property $(*)$ and $(A,B)$ satisfies property UC. Then $T$ is a BPO on $[x]_{\tilde{G}}\cap A$ for $x \in X_{T^2}^A$.
 \end{theorem} 
 \begin{proof}
 Suppose $x\in X_{T^2}^A$ and $y\in [x]_{\tilde{G}}\cap A.$ Since $A$ has property $(*),$ we have $x,y\in X_{T^2}^A$ and $G$ in $[x]_{\tilde{G}}\cap A$ is weakly connected. By Theorem \ref{equivalence theorem}, $\{T^{2n}x\}$ and $\{T^{2n}y\}$ are Cauchy equivalent. Let $x^*, y^*\in A$ such that $T^{2n}x\to x^*$ and $T^{2n}y\to y^*.$ Again by property $(*),$ for $n\in \mathbb{N},~(T^{2n}x, x^*)$ and $(T^{2n}y, y^*)$ are in $E(G)$ and hence $(x, x^*)$ and $(y, y^*)$ are in $E(G).$ But $(x,y)\in E(\tilde{G}).$ Then $(y, x^*)$ and $(x,y^*)$ are in $E(\tilde{G}).$ Thus $(x^*,y^*)\in E(\tilde{G})$ and hence $x^*, y^*\in [x]_{\tilde{G}}\cap A.$ Since $d(T^{2n+1}x, x^*)\leq d(T^{2n}x,T^{2n+1}x)+d(T^{2n}x,x^*)\to d(A,B)$ and  we see that 
 \begin{eqnarray*}
 d(x^*, Tx^*)&\leq & d(x^*, T^{2n+2}x)+ d(T^{2n+2}x, Tx^*)\\
 &\leq & d(x^*, T^{2n+2}x)+ d(T^{2n+1}x, x^*)\to d(A,B).
 \end{eqnarray*}
 Thus, $d(x^*, Tx^*)=d(A,B)$ and analogously one can show that $d(y^*, Ty^*)=d(A,B).$ 
  Suppose $x^*\neq y^*.$ Then $\min\{d(x^*, Ty^*), d(y^*, Tx^*)\}>d(A,B).$ Since $(y^*,T^2x^*), (x^*,T^2y^*)\in E(\tilde{G})$, we have 
 \begin{eqnarray*}
 d(x^*, Ty^*)&=& d(T^2x^*, Ty^*)\\
  &\leq & (I-\phi_1)(d(y^*, Tx^*))+(I-\phi_2)(m(y^*, Tx^*)+(\phi_1+\phi_2-I)(d(A,B))\\
  &\leq & (I-\phi_1)(d(y^*, Tx^*))+(I-\phi_2)(d(y^*, Ty^*)+(\phi_1+\phi_2-I)(d(A,B))\\
  &< & d(y^*, Tx^*)-\phi_1(d(y^*, Tx^*))+(I-\phi_2)(d(A,B)+\phi_1(d(y^*, Tx^*)) + (\phi_2-I)(d(A,B))\\
 &= &d(y^*, Tx^*).
 \end{eqnarray*}   
   Interchanging $x^*$ and $y^*,$ we get $d(y^*, Tx^*)
  <d(x^*, Ty^*).$ This is absurd. Hence $x^*=y^*.$ This completes the proof.
 \qed \end{proof}

 As an immediate consequence of the above theorem we get
 \begin{corollary}\label{corollary of main theorem 1}

   Let $A, B, T$ and $(X,d,G)$ be as in Theorem \ref{main theorem1}.
  If $G$ is weakly-connected, then $T$ has a unique best proximity point $x^*$ in $A$ and $\displaystyle \lim_{n \to \infty} T^{2n}x=x^*$ for any $x \in A.$
 \end{corollary}

 The following example illuminates Theorem \ref{main theorem1}.
 \begin{example}\label{example of main theorem 1}
 Let $X=\mathbb{R}^2$ with $\|\cdot\|_{1}$ norm and $A=\{(0, x)\in \mathbb{R}^2: x\in \{0, \frac{1}{2}, \frac{1}{2^2}, \frac{1}{2^3},...\}\}, B=\{(1,y)\in \mathbb{R}^2: y\in \{0,\frac{1}{2}, \frac{1}{2^2}, \frac{1}{2^3},...\}\}$. Consider a graph $G$ on $X$ such that $V(G)=X$ and $E(G)=\{(x,y)\in A\times A: d(x,y)\leq \frac{1}{2}\}\cup \{(x,y)\in B\times B: d(x,y)\leq \frac{1}{2}\}\cup \{(x,y)\in X\times X: d(x,y)=1\}.$ 
  Observe that $A$ is complete, $A$ has property $(*)$ and $(A,B)$ has property UC. Moreover, $G$ in $A$ is weakly connected and for any $x\in A,~ [x]_{\tilde{G}}\cap A=A$. Define a map $T$ on $X$ as the following:
\[
 T(0,x)=
 	\begin{cases}
 	(1,0) & \text{for}~ x=0\\
 	(1,\frac{1}{2^{n+1}}) & \text{for}~ x=\frac{1}{2^n}, n\geq 1
 	\end{cases};~
 	 T(1,y)=
 	\begin{cases}
 	(0,0) & \text{for}~ y=0\\
 	(0,\frac{1}{2^{n+1}}) & \text{for}~ y= \frac{1}{2^n}, n\geq 1.
 	\end{cases}
 \]
 A simple calculation shows that $T$ is a $G$-cyclic $(\phi, c+I)$-contraction with $\phi(s)=\frac{1}{2}s$ and $c$ is a constant. 
The case when $(x, y)\in A\times B\cup B\times A$ with $\|x-y\|_1=1$ is trivial (since $(A,B)$ is a sharp proximal pair we have $\|Tx-Ty\|_1=1$ and hence we are through). Let $x=(0,\frac{1}{2^n})$ and $y=(1,\frac{1}{2^{n+p}})$ for some $n, p\in \mathbb{N}$ with $(x,Ty)\in E(G).$ We have $\|Tx-Ty\|_1 = \left\|\left(0,\frac{1}{2^{n+1}}\right)-\left(1,\frac{1}{2^{n+p+1}}\right)\right\|_1 = 1+ \frac{1}{2^n}\left(\frac{1}{2}-\frac{1}{2^{p+1}}\right).$
Since, $ \|x-y\|_1-\phi(\|x-y\|_1)+\phi(d(A,B)) = \left\|\left(0,\frac{1}{2^{n}}\right)-\left(1,\frac{1}{2^{n+p}}\right)\right\|_1= 1+\left(\frac{1}{2^{n}}-\frac{1}{2^{n+p}}\right)=1+\frac{1}{2^n}\left(1-\frac{1}{2^{p}}\right),$ we have
{\small{\begin{eqnarray*}
\|x-y\|_1-\phi(\|x-y\|_1)+\phi(d(A,B)) &=& 1+\frac{1}{2^n}\left(1-\frac{1}{2^{p}}\right)-\phi\left(1+\frac{1}{2^n}\left(1-\frac{1}{2^{p}}\right)\right)+\phi(1)\\
&=& 1+\frac{1}{2^n}\left(1-\frac{1}{2^{p}}\right)-\left(\frac{1}{2}+\frac{1}{2^{n+1}}\left(1-\frac{1}{2^{p}}\right)\right)+\frac{1}{2}\\
&=& 1+\frac{1}{2^n}\left(\frac{1}{2}-\frac{1}{2^{p+1}}\right).
\end{eqnarray*}}}
  Then by Corollary \ref{corollary of main theorem 1}, $T$ is a BPO on $A.$ Observe that $BP\left(T\big|_A\right)=\{(0,0)\}.$
 \end{example}
 
 The following corollary is a direct consequence of Proposition \ref{proposition1} and Theorem \ref{main theorem1}.
 \begin{corollary}
   Let $A, B, T$ and $(X,d,G)$ be as in Theorem \ref{main theorem1}. Then $BP\left(T\big|_A\right)\neq \emptyset$ if and only if $X_{T^2}^A\neq \emptyset.$
 \end{corollary}

 The following example illustrates that, in Theorem \ref{main theorem1}, if $x\notin X_{T^2}^A,$ then $T$ is not a BPO on $[x]_{\tilde{G}}\cap A.$
 \begin{example}\label{not BPO}
 Let $X=[0,1]\times [0,1]$ be endowed with $\|\cdot\|_1$ norm and $A=\{(0,x)\in \mathbb{R}^2:0\leq x\leq 1\}, B=\{(1,y)\in \mathbb{R}^2:0\leq y\leq 1\}.$ Then $X$ is complete and $A,B$ are closed bounded subsets of $X.$ Define a relation $\preceq$ on $X$ by the following rule:\\
 for $\bar{x}=(a, x), \bar{y}=(b,y)\in X,$ we say $\bar{x}\preceq \bar{y}$ if and only if $a,b\in [0,1]$ and either $x=\frac{y}{2}$ or $y=\frac{x}{2}$ or $x=y.$ We say $(\bar{x}, \bar{y})\in E(G)$ if $\bar{x}\preceq \bar{y}.$

Define a map $T:A\cup B\to A\cup B$ by
 \[
 T(0,x)=
 	\begin{cases}
 	(1,\frac{x}{2}) & \text{for}~ x\in (0,1);\\
 	(1,0) & \text{for}~ x=0;\\
 	(1,1) & \text{for}~ x=1.
 	\end{cases}~
 	 T(1,y)=
 	\begin{cases}
 	(0,\frac{y}{2}) & \text{for}~ y\in (0,1);\\
 	(0,0) & \text{for}~ y=0;\\
 	(0,1) & \text{for}~ y=1.
 	\end{cases}
 \] 

  One can verify that $T$ is a $G$-cyclic $(\phi, c+I)$-contraction with $\phi(s)=\frac{1}{2}s$ and $c$ is a constant. 
  Let $(\bar{x},\bar{y})\in A\times B$ such that $(\bar{x},\bar{y})\in E(G),$ where $\bar{x}=(0,x)$ and $\bar{y}=(0,y)$ with $x,y\in (0,1).$ Then $\|\bar{x}-\bar{y}\|_1=1+|x-y|$ and $\|T\bar{x}-T\bar{y}\|_1=1+\frac{|x-y|}{2}= 1+|x-y|-\left(\frac{1}{2}+\frac{|x-y|}{2}\right)+\frac{1}{2}=\|\bar{x}-\bar{y}\|_1-\phi(\|\bar{x}-\bar{y}\|_1)+\phi(1).$ For $((0,0),(1,0))\in E(G),$ we have $(T(0,0),T(1,0))\in E(G)$ and $\|T(0,0)-T(1,0)\|_1=\|(0,0)-(1,0)\|_1.$ Also, $((0,1),(1,1))\in E(G),$ we have $(T(0,1),T(1,1))=((1,1),(0,1))\in E(G)$ and $\|T(0,1)-T(1,1)\|_1=1=\|(0,1)-(1,1)\|_1.$ Let $(\bar{x},\bar{y})\in A\times B$ such that $(\bar{x},T\bar{y})\in E(G)$ or $(T\bar{y},\bar{x})\in E(G)$ where $\bar{x}=(0,x)$ and $\bar{y}=(0,y)$ with $x,y\in (0,1).$ 
 A similar computation shows that $\|T\bar{x}-T\bar{y}\|_1\leq\|\bar{x}-\bar{y}\|_1-\phi(\|\bar{x}-\bar{y}\|_1)+\phi(1).$ We have $((0,0),T(1,0))$ and $((0,1),T(1,1))$ are in $E(G)$ and $\|T\bar{x}-T\bar{y}\|_1\leq \|\bar{x}-\bar{y}\|_1-\phi(\|\bar{x}-\bar{y}\|_1)+\phi(1)$ for $\bar{x}=(0,0), \bar{y}=(1,0)$ and $\bar{x}=(0,1), \bar{y}=(1,1).$ We observe that $((0,0),T(1,1))=((1,0),(0,1))\notin E(G)$ and $((0,1),T(1,0))=((1,1),(0,0))\notin E(G).$ And for this $\|T(0,0)-T(1,1)\|_1=\|(1,0)-(0,1)\|_1=2>2-1+\frac{1}{2}=\|(0,0)-(1,1)\|_1-\phi(\|(0,0)-(1,1)\|_1)+\phi(1).$
   
   For $x\in (0,1), \bar{x}=(0,x)\in A$ and $[\bar{x}]_{\tilde{G}}\cap A=\{(0,z): \exists n\in \mathbb{N}~\mbox{such that}~z=\frac{x}{2^{n-1}}~\mbox{or}~x=\frac{z}{2^{n-1}}\}.$ But $T|_{[\bar{x}]_{\tilde{G}}\cap A}$ is not a $BPO.$  Observe that $BP\left(T\big|_A\right)=\{(0,0), (0,1)\}\nsubseteq [\bar{x}]_{\tilde{G}}.$ It is worth mentioning that $\bar{x}\notin X^A_{T^2}$ for $x\in (0,1).$ 
 \end{example}
 
 Theorem \ref{on cardinality} ensures the number of connected subgraphs in $G$ is same with the number of best proximity points of $T$ in $A.$
 \begin{theorem}\label{on cardinality}
   Let $A, B, T$ and $(X,d,G)$ be as in Theorem \ref{main theorem1}. Then $\# BP\left(T\big|_A\right)= \# E_A$, where $\#$ denotes the cardinality of the corresponing set.
 \end{theorem}
 
 \begin{proof}
 Define $P: BP\left(T\big|_A\right)\to E_A$ by $Px=[x]_{\tilde{G}}$ for $x\in BP\left(T\big|_A\right).$ 
   Let $x_1, x_2\in BP\left(T\big|_A\right)$ such that $Px_1=Px_2.$ Then $x_1, x_2\in X_{T^2}^A$ and $x_2\in [x_1]_{\tilde{G}}$. By Theorem \ref{main theorem1}, $T|_{[x_1]_{\tilde{G}}\cap A}$ is a BPO. 
   Thus, $x_2=\displaystyle \lim_{n\to \infty}T^{2n}x_2 =\lim_{n\to \infty}T^{2n}x_1=x_1.$
Hence $P$ is injective.
 To see $P$ is surjective, let $x\in X_{T^2}^A.$ Then $(x,T^2x) \in E(G).$ 
  By Theorem \ref{main theorem1}, there exists a best proximity point $x^*$ of $T$ in $A$ with $\displaystyle \lim_{n\to \infty} T^{2n}x = x^*$. 
   By property $(*),$ we have $[x]_{\tilde{G}}=[x^*]_{\tilde{G}}$. Thus $Px^*=[x]_{\tilde{G}}.$ This completes the proof. 
 \qed \end{proof}
 
 We illuminate the above theorem by the following example:
 \begin{example}
 Let $A, B, T$ and $(X,d,G)$ be as in Example \ref{not G-contraction 3}. One can see that $ A$ is complete, has property $(*)$ and $(A, B)$ satisfies property UC. A simple numerical computation leads that $BP\left(T\big|_A\right)=\left\{f_{\frac{1}{n}}: n\in \mathbb{N}\right\}\cup \{f_0\}$ and $E_A=\displaystyle \left\{\left[f_{\frac{1}{n}}\right]: n\in \mathbb{N}\right\}\cup \{\left[f_0\right]\}.$
 \end{example}
 

\section{Fixed point theorems for $G$-cyclic contractions}

Let $(A,B)$ be a $G$-Chebyshev pair in a metric space with a graph $(X,d,G).$
Suppose $\phi:[0,\infty)\to [0,\infty)$ is an increasing map and $c$ is a constant. Assume that $T$ is a $G$-cyclic $(\phi, c+I)$-contraction on $A\cup B.$
 If we replace the non-emptyness of the set $X_{T^2}^A$ by the non-emptyness of $X_T^A$, then it is easy to verify that the conclusions of the above all results are valid for the cyclic map $T$ that satisfies
\begin{itemize}
 \item[(i)] $T$ preserves the edges on $A;$
  \item[(ii)] $d(Tx,Ty)\leq d(x,y)-\phi(d(x,y))+ \phi(d(A,B))$ for every edge $(x,y)\in A\times B\cup B\times A.$
 \end{itemize}

In this section we aim to prove a few fixed point theorems for such a $T.$
 A necessary condition for the existence of a fixed point of $T$ is $d(A,B)=0.$ Thus $d(Tx,Ty)\leq (I-\phi)(d(x,y))+\phi(0)$ for every egde $(x,y)\in A\times B\cup B\times A.$ This motivates us to define a type of contraction as follows:

 \begin{definition}
 Let $A_1,~A_2$ be a $G$-Chebyshev pair of non-empty subsets a metric space with a graph $(X,d,G)$ and let $T_1: A_1\to A_2, T_2:A_2\to A_1$ be two mappings. Suppose $\psi:[0,\infty)\to [0,1)$ is a non-decreasing map. The pair $(T_1,T_2)$ is said to be a $G$-$\psi$-contraction if
 $(T_ix,T_jT_ix)\in E(G)$ and $d(T_ix,T_jT_ix)\leq \psi (d(x,T_ix))d(x,T_ix)$ for $x\in A_i$ with $(x,T_ix)\in E(G)$ and $i\neq j \in \{1,2\}$.
 \end{definition}

 It has to be observed that if $T$ is $G$-cyclic $(\phi, c+I)$-contraction on $A\cup B$ for an increasing map $\phi: [0, \infty)\to [0, \infty)$ and a constant $c$, then $(T|_A, T|_B)$ is $G$-$\psi$-contraction, where
 $t(1-\psi(t))=\phi(t)-\phi\left(d(A,B)\right)$ for all $t\in \left[d(A,B),\infty\right).$ 
The following theorem ensures the existence of a common fixed point for a $G$-$\psi$-contraction.
 \begin{theorem}\label{fixedpoint 1}
 Let $(A,~B)$ be a pair of two non-empty subsets a metric space with a graph $(X,d,G)$ and $(T_1,T_2)$ is a $G$-$\psi$-contraction on $A\cup B$ for a non-decreasing function $\psi:[0,\infty) \to [0,1).$ Suppose $A$ is complete and $A\cup B$ has property $(*)$. If $X^A_{T_1}\neq \emptyset$, then $A\cap B\neq \emptyset$ and $(T_1,T_2)$ has a common fixed point $p$ in $A\cap B.$ Further for any $x_0\in X^A_{T_1}, ~\{x_{n}\}$ converges to $p$, here $x_{2n}=(T_2\circ T_1)^nx_0$ and $x_{2n+1}=T_1x_{2n}, n\geq 0$.
 \end{theorem}
 
 \begin{proof}
 Fix an $x_0\in X^A_{T_1}.$ Set $T=T_2\circ T_1, x_{2n}=T^nx_0$ and $x_{2n+1}=T_1\circ T^nx_0.$ Then $d_n\leq (\psi(d(x_0,T_1x_0)))^nd(x_0,T_1x_0)$ for $n\geq 1$, where $d_n=d(x_{n}, x_{n+1})$. Now for any fixed $n,m\geq 1,$
 {\small{\begin{eqnarray*}
 d(x_n,x_{n+m}) &\leq & d(x_n,x_{n+1})+d(x_{n+1},x_{n+2})+\cdots+d(x_{n+m-1}, x_{n+m})\\
 &\leq & \left(\psi(d(x_0,x_1))\right)^nd(x_0,x_1)
 +\cdots +~ (\psi(d(x_0,x_1)))^{n+m-1}d(x_0,x_1)\\
 &=& (\psi(d(x_0,x_1)))^nd(x_0,x_1) \left[1+\psi(d(x_0,x_1))
 +\cdots + ~(\psi(d(x_0,x_1)))^{n+m-2}\right]\\
 &\leq & (\psi(d(x_0,x_1)))^nd(x_0,x_1) [1+\psi(d(x_0,x_1))+(\psi(d(x_0,x_1)))^2+\cdots]\\
 &=& \frac{(\psi(d(x_0,x_1)))^nd(x_0,x_1) }{1-\psi(d(x_0,x_1))}.
 \end{eqnarray*}}}
 Hence, $\{x_{n}\}$ is Cauchy. Since $A$ is complete, there exists $p\in A$ such that $x_{2n}\to p.$ Thus $x_{2n+1}\to p$.
 Therefore, $(x_{2n}, x_{2n+2}), (x_{2n},p), (x_{2n+1}, p)$ in $E(G)$ for $n\geq 0.$ Now, $p\in A$ implies $T_1p\in B$ and 
  \begin{eqnarray*}
 d(x_{2n+2}, T_1p)\leq \psi(d(x_{2n+1}, p))d(x_{2n+1}, p) \leq d(x_{2n+1}, p)\to 0.
  \end{eqnarray*}
  Then $p=T_1p$ and hence $p\in A\cap B.$ Also, 
  \begin{eqnarray*}
  d(x_{2n+1}, T_2p)&\leq & \psi(d(x_{2n}, p))d(x_{2n}, p)\leq d(x_{2n}, p)\to 0.
  \end{eqnarray*}
  Thus $p=T_2p.$ Hence $p$ is a commom fixed point of $(T_1, T_2).$
 \qed \end{proof}
 Let's suppose the $G$-$\psi$-contraction pair $(T_1,T_2)$ satisfies $(T_ix,T_jT_iy)\in E(G)$ and $d(T_ix,T_jT_iy)\leq \psi (d(x,T_iy))d(x,T_iy),~(1 \leq i\neq j \leq 2),$ for $x, y$ in $A_i$ with $(x,T_iy) \in E(G).$ 
 Further assuming $G$ in $A$ is weakly connected, in Theorem \ref{fixedpoint 1},
   we get the uniqueness of the common fixed point. For this, suppose $p, q$ are two common fixed points of $(T_1,T_2)$. Then $(p,q) \in E(G)$ and so $d(p, q)=d(T_1p, T_2q)\leq \psi(d(p,q)) d(p,q)$. This is possible only if $p=q$. We say that $G$ is a weak friendship graph on $A$, if for every pair of points in $A$ has a common neighbor in $A$ (i.e., for $x,y\in A$, there exists $u \in A$ such that $(u,x)$ as well as $(u,y)$ in $E(G))$. Suppose $p, q$ are two common fixed points of $(T_1, T_2)$. Then there exists $r \in A$ such that $(r,p)$ and $(r,q)$ in $E(G)$. Thus $(Tr,p) \in E(G)$ for $n \geq 0$, where $T=T_2\circ T_1$. Hence $d(p,T^nr) \leq \psi(d(p,r))^nd(p,r)$ for $n \geq 0$. In a similar way, we have $d(q,T^nr) \leq \psi(d(q,r))^nd(q,r)$ for $n \geq 0$. Therefore we have $d(p,q)=0$. Thus, we have the following uniqueness theorem.
 
 \begin{theorem}\label{uniqueness of fixed points}
 Adding either $G$ is weakly connected or weak friendship graph on $A$ to Theorem \ref{fixedpoint 1}, we obtain uniqueness of the common fixed point of $(T_1, T_2)$, here for $x, y$ in $A_i$ with $(x,T_iy) \in E(G)$,
  $(T_ix,T_jT_iy)\in E(G)$ and $d(T_ix,T_jT_iy)\leq \psi (d(x,T_iy))d(x,T_iy),~(1 \leq i\neq j \leq 2).$ 
 \end{theorem}
 
  In \cite{Jachymski 2008}, the authors introduced a type of continuous mappings on a similar setting and therein proved the existence of a fixed point. A map $T$ on a metric space with a graph $(X,d,G)$ is called orbitally $G$-continuous if for all $x, p\in X$ and any sequence $\{k_j\}$ of natural numbers, 
 $T^{k_j}x\to p$ and $(T^{k_j}x,T^{k_{j+1}}x)\in E(G)$ for $j\in \mathbb{N}$ imply $T(T^{k_j}x)\to Tp.$ 

  It is to be observed that the assumption $``A\cup B$ has property $(*)$" in Theorem \ref{fixedpoint 1} can be replaced by ``$T$ on $A_1\cup A_2$ is orbitally $G$-continuous", here $Tx\in A_2$ if $Tx=T_1x$ and $Tx\in A_1$ if $Tx=T_2x$ for $x\in A\cup B$. Thus, one can obtain Theorem 2.4 of \cite{Neito 2005} as a consequence of Theorem \ref{uniqueness of fixed points}.

 Now, we illustrate Theorem \ref{uniqueness of fixed points} with an example.
  
 \begin{example}\label{example of fixed point 2}
 Suppose $X=\mathscr{C}[0,1],$ the linear space of all complex valued
continuous functions over the interval $[0,1]$ with the norm $\|f\|=\max\left\{\displaystyle\|f_1\|_\infty, \|f_2\|_\infty\right\},$ where $f=f_1+if_2,$ $f_j$ is a real valued function $(1\leq j\leq 2).$ Set $A=\{f_1+if_2\in X: 0\leq f_1(t)\leq 1, f_2(t)=0~\mbox{for}~t\in [0,1]\}$ and $B=\{g_1+ig_2\in X: g_1(t)=0,~0\leq g_2(t)\leq 1 ~\mbox{for}~t\in [0,1]\}.$ Consider a graph $G$ on $X$ such that $V(G)=X$ and $E(G)=\{(f,g)\in A\cup B\times B\cup A: d(f,g)<1\}.$ Let $\psi: [0,\infty)\to [0, 1)$ be a non-decreasing map. Define $T_1:A\to B$ and $T_2:B\to A$ for $f=f_1+if_2\in A\cup B,$ by
 \[
 	T_1f=
 	\begin{cases}
		f_2+i \psi \left(\displaystyle\|f_1\|_\infty\right) f_1 ~&\text{if}~f=f_1+if_2\in A, 0\leq f_1(t)<1, \\
 		f_2+i~&\text{if}~f=f_1+if_2\in A, f_1(t)=1;
 	\end{cases}
 \] 
 \[
 	T_2f=
 	\begin{cases}
 		 \psi \left(\displaystyle\|f_2\|_\infty \right) f_2 +if_1~&\text{if}~f=f_1+if_2\in B, 0\leq f_2(t)<1\\
 		f_1+if_1~&\text{if}~f=f_1+if_2\in B, f_2(t)=1.
 	\end{cases}
 \]
 
 A simple computation leads us that $(T_1,T_2)$ is a cyclic $G$-$\psi$-contraction on $A\cup B,$ whereas $T$ is not a cyclic contraction, here $Tf=T_1f$ if $f\in A$ and $Tf=T_2f$ if $f\in B.$ It's clear that $A$ is complete, $A\cup B$ has property $(*)$ and for $0\leq \alpha<1,$ the constant function $\alpha\in A$ satisfies $\left(\alpha, T_1(\alpha)\right)\in E(G)$. Also, $G$ is weakly connected. Hence by Theorem \ref{fixedpoint 1}, $(T_1, T_2)$ has a unique common fixed point, viz., the constant function $0.$ 

 \end{example} 

 \section{Application to the PBVPs}
  
 Let $I=[0,T], T>0$ be a time interval and for a set $J$, $L^1(J)$ and $AC(J)$ denote the set of all real valued integrable and absolutely conintuous functions on $J$ respectively.
 For two functions $u,v \in \mathscr{C}(I)$, the real valued continuous function space on $I$ with the supremum norm, we say that $u \leq v$, if $u(t) \leq v(t)$ for all $t \in I$.
 Consider the following nonlinear periodic boundary value problem (abbrev. PBVP):
 \begin{eqnarray}
 u'(t)&=&f_1(t,u(t)), ~u(0)=u(T)\label{initial equation}
 \end{eqnarray}
 for $t\in I$ and $f\in L^1(I\times \mathbb{R}).$ A function $w\in AC(I)$ is said to be a lower (respectively, an upper) solution  for (\ref{initial equation}) if 
 \begin{eqnarray*}
 w'(t)\leq f_1(t,w(t))~\mbox{a.e.}~t\in I~\mbox{and}~ w(0)\leq w(T)
 \end{eqnarray*}
 (respectively, $w$  satisfies the reverse inequalities). The method of lower and upper solutions have undergone an effective mechanism in the existence of a solution ($w\in AC(I)$ that satisfies (\ref{initial equation})) of (\ref{initial equation}), for more details the reader can refer to 
 [13-16].
  The following is a canonical form of a solution for the PBVP (\ref{initial equation}). 

 \begin{lemma}\cite{Lakshmikantham}\label{Lakshmikantham}
 Let $\mathbb{E}$ be a Banach space and $f:I\times \mathbb{E}\to \mathbb{E}.$ A function $u:I\to \mathbb{E}$ is a solution of the periodic boundary value problem (\ref{initial equation}) on $I$ if and only if $u$ satisfies for any $p\in L^1(I,\mathbb{R}),$ with $P(t)=\displaystyle \int_{s=0}^{t}p(s)ds$ nonzero at $t=T,$ the integral equation
 \begin{eqnarray*}
 u(t)&=& e^{-P(t)} \displaystyle \int_{s=0}^{t} e^{P(s)}[f_1(s,u(s))+p(s)u(s)]ds \\ &+& \frac{e^{-P(t)}}{e^{P(T)}-1} \int_{s=0}^{T} e^{P(s)}[f_1(s,u(s))+p(s)u(s)]ds.
 \end{eqnarray*}
 \end{lemma}
 
   Let's contemplate the following system of PBVPs:
 \begin{eqnarray}
 u'(t)=f_1(t,u(t)), ~u(0)=u(T);~
 v'(t)=f_2(t,v(t)),~v(0)=v(T),\label{system 1}
 \end{eqnarray}
 for $t\in I$ and $f_1,f_2\in L^1(I\times \mathbb{R}).$
 Here our aim is to find a common solution for the system (\ref{system 1}). This system can be re-written as 
 \begin{eqnarray}
 u'(t)+ \alpha u(t)&=& f_1(t, u(t))+ \alpha u(t)~\mbox{a.e.}~t\in I~\mbox{and}~ u(0)= u(T),\label{system 1.1}\\
 v'(t)+ \alpha v(t)&=& f_2(t, v(t))+ \alpha v(t)~\mbox{a.e.}~t\in I~\mbox{and}~ v(0)= v(T)\label{system 1.2}
 \end{eqnarray}  
 for an $\alpha\in \mathbb{R}.$
 Fix an $\alpha>0.$
For $u,v\in AC(I)$, one can see by Lemma \ref{Lakshmikantham} that $(u, v)$ is a solution of the system (\ref{system 1}) if and only if 
 \begin{eqnarray*}
 u(t)=\displaystyle \int_{s=0}^{T}G(t,s)[f_1(s,u(s))+\alpha u(s)]ds;~~
 v(t)=\displaystyle \int_{s=0}^{T}G(t,s)[f_2(s,v(s))+\alpha v(s)]ds,
 \end{eqnarray*}
where
 \[
 G(t,s)=
 \begin{cases}
 \frac{e^{\alpha (T+s-t)}}{e^{\alpha T}-1} & \text{if}~  0\leq s< t\leq T,\\
 \frac{e^{\alpha (s-t)}}{e^{\alpha T}-1} & \text{if}~  0\leq t< s\leq T.
 \end{cases}
 \]

  Consider a graph $G$ on $X=\mathscr{C}(I)$ with the vertex set $V(G)=X$ and the edge set $E(G)=\{(x,y)\in X\times X: x\leq y\}.$ Set
 {\small{\begin{eqnarray*}
 A_1 &=& \{x\in X: x(0)=x(T), x(t)=\displaystyle\int_{s=0}^{T}G(t,s)[f_1(s,y(s))+\alpha y(s)]ds~\mbox{for some}~y\in X\},\\
 B_1 &=& \{x\in X: x(0)=x(T), x(t)=\displaystyle\int_{s=0}^{T}G(t,s)[f_2(s,y(s))+\alpha y(s)]ds~\mbox{for some}~y\in X\}.
 \end{eqnarray*}}}
  To see $A_1$ is non-empty, let $y(s)=k$ be a constant function on $I$ for some $k\in \mathbb{R}.$ Then $x(t)=\displaystyle\int_{s=0}^{T}G(t,s)[f_1(s,k)+\alpha k]ds$ is an absolutely continuous function and since $G(0,s)=G(T,s)$ for $s\in I,$ we have  $x(0)=\displaystyle\int_{s=0}^{T}G(0,s)[f_1(s,k)+\alpha k]ds=\displaystyle\int_{s=0}^{T}G(T,s)[f_1(s,k)+\alpha k]ds=x(T).$ Hence $x\in A_1.$ Similarly, $B_1\neq \emptyset.$
 The following theorem ensures the existence of a common solution for the system (\ref{system 1}) if $f_i$ is a Carath\'{e}odory function, $i=1, 2.$ A function $f:I\times \mathbb{R}\to \mathbb{R}$ is said to be Carath\'{e}odory if $f(\cdot,s)$ is measurable for each fixed $s\in \mathbb{R}$ and $f(t,\cdot)$ is continuous for each fixed $t\in I.$ 
 \begin{theorem}\label{differential theorem}
 Let us consider the system (\ref{system 1}). Assume the following:
 \begin{itemize}
 \item[(i)] $f_i$ is a Carath\'{e}odory function, $i=1,2;$
 \item[(ii)] there exists $L^1$ functions $m_1, m_2$ on $I$ such that $|f_i(t,\cdot)|\leq m_i(t)$ for $t\in I$ and $i=1, 2;$
 \item[(iii)] there exists a lower solution $w$ for (\ref{system 1.1});
 \item[(iv)] there exist $\alpha>0$ and a function $h:I\to (0,\infty)$ with $\displaystyle\sup_{t\in I}h(t)<\alpha$ such that $|f_1(t,s_2)+\alpha s_2-\left(f_2(t,s_1)+\alpha s_1\right)|\leq h(t) (s_2-s_1)$ ~for almost all $t\in I$ and $s_1, s_2\in \mathbb{R}$ with $s_1\leq s_2;$ 
 \item[(v)] if $x(t)\leq y(t)=\displaystyle\int_{0}^{T}G(t,s)\left(f_i(s,x(s))+\alpha x(s) \right)ds$ for some $x\in X$ and $i\in \{1,2\},$ then $y(t)\leq \displaystyle\int_{0}^{T}G(t,s)\left(f_j(s,y(s))+\alpha y(s) \right)ds$ for $j\in \{1,2\}$ and $i\neq j.$
 \end{itemize}
 Then the system (\ref{system 1}) has a common solution.
 \end{theorem}  

 \begin{proof}
 Let $X, A_1, B_1$ and $G$ be as above. Set $A=cl(A_1)$ and $B=cl(B_1),$ where $cl(A_1)$ and $cl(B_1)$ are the closures of $A_1$ and $B_1$ respectively. It's easy to verify that $A\cup B$ has property $(*).$ Define $F_1:A\to B$ and $F_2: B\to A$ By
 {\small{\[ 
 	F_1u(t)= \displaystyle\int_{s=0}^{T}G(t,s)[f_2(s,u(s))+\alpha u(s)]ds;~
	F_2u(t)= \displaystyle\int_{s=0}^{T}G(t,s)[f_i(s,u(s))+\alpha u(s)]ds
 \]}} 
respectively. Since $G(0,s)=G(T,s),~F_1(A_1)\subseteq B_1$ and $F_2(B_1)\subseteq A_1.$
 For $x\in A\setminus A_1$ there exists a sequence $\{x_n\}$ in $A_1$ such that $\|x_n-x\|\to 0$. Therefore, for $n\geq 1,$ $F_1x_n\in B_1$ and there is $y_n\in X$ such that $x_n(t)=\displaystyle\int_{s=0}^{T}G(t,s)[f_1(s,y_n(s))+\alpha y_n(s)]ds.$ For given $\epsilon >0$, there exists $N\in \mathbb{N}$ such that for $n\geq N,~\left|\left(f_2(s,x_n(s))+\alpha x_n(s)\right)-\left(f_2(s,x(s))+\alpha x(s)\right)\right|<\alpha \epsilon.$ Then for $n\geq N$, we have
 {\small{\begin{eqnarray*}
 |F_1x_n(t)-F_1x(t)| &\leq & \displaystyle\int_{s=0}^{T}G(t,s)\left|\left(f_2(s,x_n(s))+\alpha x_n(s)\right)-\left(f_2(s,x(s))+\alpha x(s)\right)\right|ds\\
 &\leq & \alpha \epsilon~ \displaystyle\sup_{t\in I} \int_{s=0}^{T}G(t,s) ds=\frac{\alpha \epsilon}{\alpha}=\epsilon.
 \end{eqnarray*}}}
 Thus, that $F_1x\in B.$ Hence $F_1(A)\subseteq B.$ In a similar fashion, one can show that $F_2(B)\subseteq A.$ 
 Let $x\in A$ with $(x,F_1x)\in E(G).$ By the assumption $(iv),$ we have for $t\in I,$ 
 \begin{eqnarray*}
 F_1x(t)&= &\displaystyle\int_{s=0}^{T}G(t,s)[f_2(s,x(s))+\alpha x(s)]ds\\
 &\leq &\displaystyle\int_{s=0}^{T}G(t,s)[f_2(s,F_1x(s))+\alpha F_1x(s)]ds\\
 &=& F_2F_1x(t).
 \end{eqnarray*}
 Hence $(F_1x,F_2F_1x)\in E(G).$ Similarly, one can prove that if $y\in B$ such that $(y,F_2y)\in E(G),$ then $(F_2y, F_1F_2y)\in E(G).$
A simple calculation leads that if $w$ is a lower solution of (\ref{system 1.1}), then $w(t)\leq x_0(t),$ where  $x_0(t)=\displaystyle\int_{s=0}^{T}G(t,s)[f_1(s,w(s))+\alpha w(s)]ds$ for $t\in I.$ By $(iv),$ for $t\in I$
 we have  $x_0(t)\leq \displaystyle\int_{s=0}^{T}G(t,s)[f_2(s,x_0(s))+\alpha x_0(s)]ds=F_1x_0(t).$ Thus, $(x_0,F_1x_0)\in E(G)$.
Now, for $x\in A$ with $(x,F_1x)\in E(G),$ we have, for $t\in I,$
 {\small{\begin{eqnarray*}
 |F_1x(t)-F_2F_1x(t)| &\leq & \displaystyle\int_{s=0}^{T}G(t,s)\left|(f_2(s,x(s))+\alpha x(s))-(f_1(s,F_1x(s))+\alpha F_1x(s))\right|ds\\
 &\leq & \displaystyle\int_{s=0}^{T}G(t,s) h(s)\left|x(s)- F_1x(s)\right|ds\\
 &\leq & \|x-F_1x\|_{\infty} \displaystyle\int_{s=0}^{T}G(t,s) h(s)ds\\
 &\leq & \displaystyle\sup_{t\in I} h(t) \|x-F_1x\|_{\infty} \cdot \frac{1}{\alpha}\\
 &=& \beta \|x-F_1x\|_{\infty}.
 \end{eqnarray*}}}
 Here $\beta= \frac{\displaystyle\sup_{t\in I} h(t)}{\alpha} <1.$
 Analogously, if $y\in B$ such that $(y,F_2y)\in E(G),$ then $\|F_2y-F_1F_2y\|_{\infty}\leq \beta \|y-F_2y\|_{\infty}.$ Hence, by Theorem \ref{fixedpoint 1}, $(F_1,F_2)$ has a common fixed point, say $u.$ In view of Lemma \ref{Lakshmikantham}, $u$ is a common solution of the system (\ref{system 1}).
 \end{proof}
 
 In the above proof, the assumption on continuity of $f_1$ and $f_2$ is used to show $F$ is cyclic. Further, if $f_1=f_2,$ then one can relax the aformentioned assumtion by considering $A=B=X$. In this case, $G$ is a weak friendship on $X.$ Thus using Theorem \ref{uniqueness of fixed points} and Theorem \ref{differential theorem} we have the following:
 \begin{theorem}\label{differential corollary}
 Consider the system (\ref{initial equation}) with $f(t,s)$ Carath\'{e}odory and bounded in $t\in I$. Suppose there exist $\alpha>0$, a function $h:I\to (0,\infty)$ such that $\displaystyle\sup_{t\in I}h(t)<\alpha$ and $0\leq f(t,s_2)+\alpha s_2-(f(t,s_1))+\alpha s_1)\leq h(t)(s_2-s_1)$ for $s_1, s_2\in \mathbb{R}$ with $s_1 \leq s_2$. If (\ref{initial equation}) has a lower solution, then there exists a unique solution for the same.
 \end{theorem}
 \begin{proof}
 It is sufficient to prove that $Fx\leq Fy$ and $\|Fx-Fy\|\leq \beta \|x-y\|$ for $x,y\in X$ with $x\leq y.$ Here $Fx(t)=\displaystyle\int_{0}^{T}G(t,s)(f(s,x(x))+\alpha x(s))ds$ for $t\in I$ and $\beta=\displaystyle \frac{\sup_{t\in I}h(t)}{\alpha}.$ This follows from $0\leq f(t,s_2)+\alpha s_2-(f(t,s_1))+\alpha s_1)\leq h(t)(s_2-s_1)$ for $s_1, s_2\in \mathbb{R}$ with $s_1 \leq s_2$. 
 \qed 
 \end{proof}
 Finally, we illustrate the above theorem with an example.
 
 \begin{example}
 Let $I=[0,1]$ and define a function $f$ on $I\times \mathbb{R}$ by $f(t,s)=-e^t s$ for all $t\in I$ and $s\in \mathbb{R}.$ Consider the following PBVP
 \begin{equation}\label{system corollary}
  \begin{aligned}
 u'(t) &= f(t,u(t))~\mbox{for all}~ t\in I,~ u(0)=u(1).
  \end{aligned}
 \end{equation} 
 Here $f$ is continuous on $I\times \mathbb{R}.$ If $w\equiv -1,$ then $w'(t)=0< e^t=-e^t(-1)=-e^t(w)=f(t,w)$ for all $t\in I,$ i.e., $w$ is lower solution of the system (\ref{system corollary}). Choose $\alpha=e^2$ and $h(t)=e^2-e^t, t\in I.$ Then for any $v\geq u,$ we have $f(t,v)+e^2v-(f(t,u)+e^2u)=-e^t(v-u)+e^2(v-u)=(e^2-e^t)(v-u)\geq 0.$ Therefore, all the hypotheses of Theorem \ref{differential corollary} are satisfied. Hence, the system (\ref{system corollary}) has a unique solution $u_0,~u_0(t)= 0$ for all $t\in I.$
 \end{example}

\begin{acknowledgements}
The authors thank Dr. Arti Pandey (Indian Institute of Technology Ropar) for introducing the notion of friendship graph to them.
\end{acknowledgements}
%


\end{document}